\documentclass[a4paper,12pt, oneside]{amsart}
\usepackage[a4paper, total={145mm, 245mm}]{geometry}
\usepackage[english]{babel}
\usepackage[T1]{fontenc}
\usepackage{amsmath,amsfonts,amssymb,mathrsfs}
\usepackage{xcolor}
\usepackage{tabularx}
\usepackage{multirow}
\usepackage{centernot}
\usepackage{hyperref,url}
\usepackage{tikz-cd}
\usepackage{parskip}
\usepackage{changepage}
\usepackage{enumerate}
\usepackage{comment}

\newtheorem{thm}{Theorem}[section] 
 
\newtheorem{prop}[thm]{Proposition} 
 
\newtheorem{lem}[thm]{Lemma} 
 
\newtheorem{cor}[thm]{Corollary} 
\newtheoremstyle{named}{}{}{\itshape}{}{\bfseries}{.}{.5em}{#3}
\theoremstyle{named} 
 
\theoremstyle{remark}
 
\newtheorem{rem}[thm]{Remark} 
\theoremstyle{definition}

\theoremstyle{remark}
 
\newtheorem{exa}[thm]{Example}

\everymath{\displaystyle}

\title[H\lowercase{igh} O\lowercase{rder} E\lowercase{lements} \lowercase{in} F\lowercase{inite} F\lowercase{ields}]{H\lowercase{igh} O\lowercase{rder} E\lowercase{lements} \lowercase{in} F\lowercase{inite} F\lowercase{ields} A\lowercase{rising} \lowercase{from} R\lowercase{ecursive} T\lowercase{owers}}

\author[D\lowercase{ose}, M\lowercase{ercuri}, P\lowercase{al}, S\lowercase{tirpe}]{V\lowercase{alerio} D\lowercase{ose}, P\lowercase{ietro} M\lowercase{ercuri}, A\lowercase{nkan} P\lowercase{al}, C\lowercase{laudio} S\lowercase{tirpe}}

\address{Valerio Dose, Department of Economics and Finance, LUISS University, Roma, Italy, ORCiD: 0000-0002-8186-0023}

\email{vdose@luiss.it}

\address{Pietro Mercuri, University of Udine, Udine, Italy, 
ORCiD: 0000-0003-4402-6432}

\email{mercuri.ptr@gmail.com}

\address{Ankan Pal, Department of Mathematics, University of L'Aquila, Italy}
\email{ankanpal100@gmail.com}

\address{Claudio Stirpe, Convitto Nazionale R. Margherita, Anagni, Italy}
\email{clast@inwind.it}



\begin{document}

\begin{abstract}
We provide a recipe to construct towers of fields producing high order elements in $\mathrm{GF}(q,2^n)$, for odd $q$, and in $\mathrm{GF}(2,2 \cdot 3^n)$, for $n \ge 1$. These towers are 
obtained recursively by
$x_{n}^2 + x_{n} = v(x_{n - 1})$, for odd $q$, or $x_{n}^3 + x_{n} = v(x_{n - 1})$, for $q=2$, where $v(x)$ is a polynomial of small degree over the prime field $\mathrm{GF}(q,1)$ and $x_n$ belongs to the finite field extension $\mathrm{GF}(q,2^n)$, for $q$ odd, or to $\mathrm{GF}(2,2\cdot 3^n)$. Several examples are carried out and analysed numerically. The lower bounds of the orders of the groups generated by $x_n$, or by the discriminant $\delta_n$ of the polynomial, are similar to the ones obtained in \cite{burkhart2009finite}, but we get better numerical results in some cases.
\end{abstract}

\maketitle

\section{Introduction}

Finding elements of high multiplicative order in a finite field is an interesting problem in computational number theory and has applications in cryptography (for instance: Discrete Logarithm Problem). A general method to find high order elements was given in \cite{gao1999elements}, later improved in \cite{conflitti2001} and \cite{popovych2014gao}. Another general result in this area is an algorithmic technique for finding primitive elements which is devised in \cite{huang2015primitive}.
Such technique is efficient in finite fields of small characteristic. 
Other strategies which allow to construct elements of high order usually address specific sequences of finite fields. In this regard, methods involving Gauss periods were first proposed in the results summarized in \cite{vonzurgathen_shparlinski2001gauss}. After that, an extensive literature followed with works such as  \cite{ahmadi_etal2010gauss}, \cite{popovych2012}, \cite{popovych2013}, \cite{chang2013gauss} and \cite{popovych2014sharpening}. Recently, Artin-Schreier extensions were also effectively used in \cite{popovych2015artin} and \cite{brochero2016artin}. 
Another interesting approach is to look for high order elements which arise as coordinates of points on an algebraic curve defined over a finite field (see for example \cite{voloch2007curves}, \cite{voloch2010elliptic} and \cite{chang2013curves}). One way which has been explored for generating elements of this type is through the iterative use of polynomial equations of type $f(x_{n - 1},x_n) = 0$, defining suitable towers of fields, which we address as \textit{recursive towers} in this work. Examples of this can be found in \cite{burkhart2009finite}, \cite{voloch2010elliptic}, \cite{popovych2015wiedmann} and \cite{popovych2018conway}. 

In \cite{burkhart2009finite}, a recursive tower defined by $f(x_{n - 1},x_n)$ is used to produce elements $\delta_n$ with high multiplicative order in $\mathrm{GF}(q,2^n)$, for $q$ odd, and in $\mathrm{GF}(q,3^n)$, for $q \neq 3$. 
The choice of the polynomial $f$ for the recursive process to generate high order elements in finite field extensions, was limited to the equations of the modular curve towers in \cite{elkies2001explicit}.

In this work, we attempt to generalize the choice of the polynomials.
We illustrate in detail several interesting towers of fields defined by $x_n^2 + x_n = v(x_{n - 1})$, where $v(x) \in \mathrm{GF}(q,1)[x]$, for $q$ odd, or $x_n^3 + x_n = v(x_{n - 1})$, for $v(x) \in \mathrm{GF}(2,1)[x]$. These towers
generate elements of high orders in $\mathrm{GF}(q,2^n)$ and in $\mathrm{GF}(2,2 \cdot 3^n)$, for $n \geq 1$. 
We also give a recipe for finding other towers of the same form which have similar properties. The simple algebraic conditions given in Sections \ref{sec:odd} and \ref{sec:even}, 
which differ 
partially from the conditions required in \cite{burkhart2009finite} (Remark \ref{BurkhartConditionsComparison} below), seem to play an important role for this purpose. In fact, in many 
of the cases we studied, these conditions are 
useful to prove the existence of high order elements
$x_n$, in the field extension. 

Throughout this paper, $\delta_n$ in $\mathrm{GF}(q,2^n)$ is the discriminant of the polynomial $f(x_{n},y)$ in $\mathrm{GF}(q,2^n)[y]$. In Corollary \ref{bound-odd}, we prove that the multiplicative orders of $x_n$ and $\delta_n$ grow very fast if $x_{n - j}^2$ and $\delta_{n-j}^2$ do not belong to $\mathrm{GF}(q,2^{n - j - 1})$, for all $j < n-1$. Similar results hold also in even characteristic, see Corollary \ref{bound-even}. Notably, despite the bounds obtained are similar, the even characteristic case turns out to be completely new
with respect to \cite{burkhart2009finite}. 
In particular no additional conditions on the discriminant are required, and the details of the proof are worked out in a different manner. Furthermore the numerical performance of some of our examples are better than \cite{burkhart2009finite}, in the iterations we were able to compute. As already mentioned above, the polynomials used in \cite{burkhart2009finite} are the models of certain modular curves given in \cite{elkies2001explicit}. Despite this fact, a possible relation of the construction of high order elements with the arithmetic properties of such curves does not seem to play a role in the proof of the lower bounds. 
Instead, in one case, we do make use of some arithmetic properties of the algebraic curve considered
by us (Lemma \ref{lem:nonsqnear5}).

A comparative study with other relevant literature has also been carried-out.
For example, a specific
construction of high order elements in the same type of fields of odd characteristic $q$ can be found in \cite{cohen1992explicit}, and some variations on it are in \cite{meyn1995explicit} and \cite{chapman1997completely}. Comparing the numerical performance of their construction with our variety of examples, we observe that the results are similar for $q \equiv 1 \pmod  4$, while for $q \equiv 3 \pmod  4$ our construction performs better (see Section \ref{sec:numres} for examples with $q=3,11$).

In Section \ref{sec:notation}, we introduce the notation that we use in the paper. In Sections \ref{sec:odd} and \ref{sec:even}, we give the main results
which allow us to obtain the lower bounds on the order of $x_n$ and $\delta_n$. 
Section \ref{sec:odd} deals with odd characteristic and Section \ref{sec:even} deals with even characteristic.
The lists of towers satisfying the properties given in Sections \ref{sec:odd} and \ref{sec:even} are provided in Sections \ref{sec:list-odd} and \ref{sec:list-even}, respectively. Finally, in Section \ref{sec:numres}, we list numerical results obtained using MAGMA \cite{MAGMA}, about the seven towers listed in Sections \ref{sec:list-odd} and \ref{sec:list-even}.

\section{Background and notation}\label{sec:notation}

Let $q$ be an odd prime. By \emph{tower of fields}, or simply a \emph{tower}, we mean a sequence of field extensions
	$$K_1\subset K_2\subset \ldots \subset K_n\subset \ldots.$$
We are interested in infinite towers, namely towers such that the degree $[K_n : K_1]$ grows to infinity. All the towers considered in this paper are actually finite, normal and separable, i.e., each extension $K_n/K_{n - 1}$ is finite, normal and separable, for every $n > 1$. For each positive integer $n$, let $K_n = \mathrm{GF}(q,1)(x_n)$, where the element $x_n \in \mathrm{GF}(q,2^n)$ is given by a recursive formula $f(x_{n - 1},x_n) = 0$, for a polynomial $f(x,y) \in \mathrm{GF}(q,1)[x,y]$. In this case, we say that the tower $K_1 \subset K_2 \subset \ldots \subset K_n \subset \ldots$ is defined by $f(x_{n - 1},x_n)$ and we address this kind of towers as \emph{recursive towers}. We focus on towers defined by $f(x_{n - 1},x_n) = x_n^2 + x_n - v(x_{n - 1})$, for $n \ge 2$, with $x_1 \in \mathrm{GF}(q,2)$, and where $v(x)$ is a polynomial in $\mathrm{GF}(q,1)[x]$. We denote by $\delta_n$ the discriminant $\delta_n = 1 + 4v(x_n)$, for $n \ge 1$. We point out that both elements $x_n$ and $\delta_n$ belong to $\mathrm{GF}(q,2^n)$, but they could also lie in a smaller extension $\mathrm{GF}(q,2^k)$ for some $k < n$. Given the tower defined by $f(x_{n-1},x_n)$, we denote by $g(x,y) \in \mathrm{GF}(q,1)[x,y]$ a polynomial giving the relation between two consecutive discriminants $\delta_{n-1}$ and $\delta_{n}$, namely $g(\delta_{n-1},\delta_{n})=0$. In the case of even characteristic (Sections \ref{sec:even} and \ref{sec:list-even}), we deal with towers defined by $f(x_{n-1},x_n) = x_n^3 + x_n + v(x_{n - 1})$, with $x_n \in \mathrm{GF}(2,2 \cdot 3^n)$, for $n \ge 1$, and $v(x)$ being a polynomial in $\mathrm{GF}(2,1)[x]$.

Given two positive integers $j$ and $n$, such that $j < n$, we denote the norm of the field extension $\mathrm{GF}(q,2^n)/ \mathrm{GF}(q,2^{n - j})$ by, $\mathrm{N}_{n,j} \colon \mathrm{GF}(q,2^n) \to \mathrm{GF}(q,2^{n - j})$. The norm in the odd case is $\mathrm{N}_{n,j}(x) = x^{\prod_{i = 1}^j (q^{2^{n - i}}+1)}$. In order to apply the same techniques to even characteristic, we also denote by $\mathrm{N}_{n,j}\colon\mathrm{GF}(2,2\cdot 3^n)\to \mathrm{GF}(2,2 \cdot 3^{n-j})$ the norm of the extension $\mathrm{GF}(2,2 \cdot 3^n)/ \mathrm{GF}(2,2 \cdot 3^{n - j})$, namely $\mathrm{N}_{n,j}(x) = x^{\prod_{i = 1}^j (4^{2 \cdot 3^{n-i}} + 4^{3^{n - i}} + 1)}$. For every characteristic, we use the conventions $\mathrm{N}(x):=\mathrm{N}_{n,1}(x)$ and $\mathrm{N}_{n,0}(x)=x$.

We use the following lemma for estimating the order of the elements in finite fields. 
\begin{lem}\label{lem:estimate}
Let $\ell$ be a prime and let
$a$, $b$ and $c$ be positive integers such that $b<c$. Assume $a\equiv 1 \bmod \ell$.
Let $p$ be a prime dividing $\frac{1}{\ell}\sum_{j=1}^{\ell}a^{\ell^b(\ell-j)}$.
Then $p>\ell^{b+1}$ and
$\gcd\left(\sum_{j=1}^{\ell}a^{\ell^b(\ell-j)},\sum_{j=1}^{\ell}a^{\ell^c(\ell-j)}\right)=\ell$. In particular $\frac{1}{\ell}\sum_{j=1}^{\ell}a^{\ell^b(\ell-j)}$ and
$\frac{1}{\ell}\sum_{j=1}^{\ell}a^{\ell^c(\ell-j)}$ are coprime.
\end{lem}
\begin{proof}
See \cite[Lemmas 1 and 2]{burkhart2009finite}.
\end{proof}
In order to compute the Galois group of a cubic polynomial in characteristic 2, the following result is useful.
\begin{lem}\label{lem:quadratic_resolvent}
Let $f(x)=x^3+ax^2+bx+c$ be a separable
irreducible polynomial
over a field $K$.
The Galois group of the extension given by the roots of $f$ is the alternating group $A_3$ if its quadratic resolvent
$R(x)=x^2+(ab-3c)x+a^3c+b^3+9c^2-6abc$
is reducible over $K$ and it is the symmetric group $S_3$ otherwise.
\end{lem}
\begin{proof}
See \cite[Section 1, pag.53]{kap1972fields}.
\end{proof}

In order to prove that a cubic polynomial is irreducible, we also need the following results.
\begin{lem}\label{lem:cubicroots1}
If $u\in \mathrm{GF}(2,2\cdot 3^{n})$
and $c:=u+u^{-1}\in \mathrm{GF}(2,2\cdot 3^{n-1})$, then $u\in \mathrm{GF}(2,2\cdot 3^{n-1})$.
\end{lem}
\begin{proof}
If $u\notin \mathrm{GF}(2,2\cdot 3^{n-1})$, then $x^2+cx+1$ is the minimum polynomial of
$u$ over $\mathrm{GF}(2,2\cdot 3^{n-1})$. So $u\in \mathrm{GF}(2,2\cdot 3^{n})\cap\mathrm{GF}(2,4\cdot 3^{n-1})=\mathrm{GF}(2,2\cdot 3^{n-1})$ and we get a contradiction.
\end{proof}
\begin{lem}\label{lem:cubicroots2}
Let $u^3\in \mathrm{GF}(2,2\cdot 3^{n-1})$ be a root of the quadratic polynomial $x^2+tx+1$, with $t\in \mathrm{GF}(2,2\cdot 3^{n-1})$. Then $y:=u+u^{-1}\in \mathrm{GF}(2,2\cdot 3^{n})$ is a root of the cubic polynomial $x^3+x+t$, and furthermore $y\in \mathrm{GF}(2,2\cdot 3^{n-1})$ if and only if $u\in \mathrm{GF}(2,2\cdot 3^{n-1})$.
\end{lem}
\begin{proof}
This is Cardano's formula for solving cubic equations in even characteristic. The second statement follows by Lemma \ref{lem:cubicroots1} taking $y=c=u+u^{-1}$.
\end{proof}

\section{Towers in odd characteristic}\label{sec:odd}

In order to find good towers we restrict our search to polynomials $f(x,y)=y^2+y-v(x)$, with $v(x)\in\mathrm{GF}(q,1)[x]$ being a non-zero polynomial, which satisfy Condition \textbf{(1)} below and at least one of the last two conditions:
\begin{description}
	\item[(1)] $\frac{f(x_{n - 1},0)}{x_{n - 1}}$ is a square in $ \mathrm{GF}(q,2^{n-1})$ for $n \geq 2$;
	\item[(2)] $\frac{g(\delta_{n - 1},0)}{x_{n - 1}}$ is a square in $\mathrm{GF}(q,2^{n-1})$ for $n \geq 2$;
	\item[(2')] $\frac{g(\delta_{n - 1},0)}{\delta_{n - 1}}$ is a square in $\mathrm{GF}(q,2^{n-1})$ for $n \geq 2$.
\end{description}
\begin{rem}\label{BurkhartConditionsComparison}
Condition \textbf{(2')} above is satisfied by other towers of fields in the literature, see for example \cite[Section 4, formula (5)]{burkhart2009finite}. We don't know whether the corresponding tower (see \cite[Section 2, equation (2)]{burkhart2009finite}), which does not satisfies Condition \textbf{(1)} above, satisfies a
suitable analog of this condition which ensure
that Proposition \ref{general-odd} below holds.
\end{rem}
\begin{rem} These conditions are not sufficient for obtaining high order elements from each tower, but, for our particular choices of $f$, they are sufficient to construct a recursive tower defined by $f(x_{n - 1},x_{n})$ as Proposition
\ref{general-odd} below shows.
\end{rem}
The following key proposition ensures that all the polynomials $f(x_{n - 1}, x_{n})$ listed in Section \ref{sec:list-odd} define infinite towers of fields. In particular it shows that $[K_n : K_{n - 1}] = 2$, for all $n > 1$. The argument of the proof is the corresponding analogue of
\cite[Proposition~1]{burkhart2009finite}
but it could be applied to many different towers.
\begin{prop}\label{general-odd}
Let $v(x) \in \mathrm{GF}(q,1)[x]$ be a polynomial and assume that $f(x_{n - 1},x_n) = x_n^2 + x_n - v(x_{n - 1})$ satisfies Conditions \textbf{(1)} \textit{and} \textbf{(2)}, or Conditions \textbf{(1)} \textit{and} \textbf{(2')}. If $x_{n - 1}$ and $\delta_{n - 1}$ are not squares in the multiplicative group $\mathrm{GF}(q,2^{n - 1})^*$ for a suitable $n \ge 2$, then $x_j$ and $\delta_j$ are not squares in the multiplicative group $\mathrm{GF}(q,2^j)^*$, for $j \ge n$.
\end{prop}
\begin{proof}
The element $x_{n}$ is not in $\mathrm{GF}(q,2^{n - 1})$ because $\delta_{n - 1}$ is not a square in $\mathrm{GF}(q,2^{n - 1})$, so $f(x_{n - 1},y)$ is the minimal polynomial of $x_n$. We need to ensure that $x_{n}^{(q^{2^n} - 1)/2} = -1$. As in \cite[Proposition 1]{burkhart2009finite}, we obtain:  
\begin{align*}
x_{n}^{(q^{2^n} - 1)/2} &= (x_{n}^{q^{2^{n - 1}} + 1})^{(q^{2^{n - 1}} - 1)/2} = \mathrm{N}(x_n)^{(q^{2^{n - 1}} - 1)/2} = \\ &= f(x_{n - 1},0)^{(q^{2^{n - 1}} - 1)/2}= -1,
\end{align*}
where $\mathrm{N}(x_n) = x_n^{q^{2^{n - 1}} + 1} = f(x_{n - 1},0)$ is the norm of $x_n$ over $\mathrm{GF}(q,2^{n - 1})$ and we use Condition \textbf{(1)} in last equality
to show that $f(x_{n - 1},0)$ is not a square in $\mathrm{GF}(q,2^{n - 1})$ for $n > 1$.

Consider the discriminant $\delta_n$. Again $g(\delta_{n - 1},y)$ is the minimal polynomial of $\delta_n = 1 + 4v(x_n)$. Since, in $\mathrm{GF}(q,2^{n})$, we know that $\frac{f(x_n,0)}{x_n}$ is a square by Condition \textbf{(1)}, $-1$ is a square and $x_n$ is not a square as above, then $v(x_n) = -f(x_n,0)$ is not a square
in $\mathrm{GF}(q,2^{n })$. 
Hence, $\delta_n \notin \mathrm{GF}(q,2^{n - 1})$. The same computation as above yields:
\begin{align*}
\delta_{n}^{(q^{2^n} - 1)/2} &= (\delta_{n}^{q^{2^{n - 1}} + 1})^{(q^{2^{n-1}}-1)/2} = \mathrm{N}(\delta_n)^{(q^{2^{n - 1}} - 1)/2}= \\ 
&= g(\delta_{n - 1},0)^{(q^{2^{n - 1}} -1)/2} = -1,
\end{align*}
where we use Condition \textbf{(2)}, respectively \textbf{(2')}, in last equality to show that $g(\delta_n,0)$ is not a square in $\mathrm{GF}(q,2^{n - 1})$, because $x_{n - 1}$, (respectively $\delta_{n - 1}$), is a non-square by hypothesis. It follows that $x_{n}$ and $\delta_{n}$ are non-squares in $\mathrm{GF}(q,2^n)$. Repeating the same argument, we find that $x_{j}$ and $\delta_{j}$ are not squares in $\mathrm{GF}(q,2^{j})$, for all $j > n$, which completes the proof.
\end{proof}

The importance of this proposition is evident if we consider Corollary~\ref{bound-odd} below, which is an analogue of \cite[Proposition 2]{burkhart2009finite}. We first state the following property of the norm  that is used in the proof of the corollary.
\begin{lem}\label{lem:factor_odd} Let $n\geq 2$ and $j<n$ be positive integers, then
$$\frac{\mathrm{N}_{n,j}(x_n)}{x_{n - j}} =\prod_{k=1}^{j}\mathrm{N}_{n-k,j-k}\left(\frac{\mathrm{N}_{n-k+1,1}(x_{n-k+1})}{x_{n - k}}\right).$$
Moreover $\frac{\mathrm{N}_{n,j}(x_n)}{x_{n - j}}$ is a square in $\mathrm{GF}(q,2^{n-j})$.
\end{lem}
\begin{proof} The case $j=1$ is trivial. By induction on $j$, let $j\geq 2$ and assume the result holds for $j-1$, then 
\allowdisplaybreaks
\begin{align*}
&\frac{\mathrm{N}_{n,j}(x_n)}{x_{n-j}}= \frac{x_n^{(q^{2^{n-1}}+1)\prod_{i=2}^{j}(q^{2^{n-i}}+1)}}{x_{n-j}}= \\
&=\left(\frac{x_n^{q^{2^{n-1}}+1}}{x_{n-1}}\right)^{\prod_{i=2}^{j}(q^{2^{n-i}}+1)}\frac{x_{n-1}^{\prod_{i=2}^{j} (q^{2^{n-i}}+1)}}{x_{n-j}}= \\
&=\left(\frac{\mathrm{N}_{n,1}(x_n)}{x_{n-1}}\right)^{\prod_{i=2}^{j}(q^{2^{n-i}}+1)}\frac{\mathrm{N}_{n-1,j-1}(x_{n-1})}{x_{n-j}}= \\
&=\left(\frac{\mathrm{N}_{n,1}(x_n)}{x_{n-1}}\right)^{\prod_{i=1}^{j-1}(q^{2^{n-1-i}}+1)}
\prod_{k=1}^{j-1}\mathrm{N}_{n-k-1,j-k-1}\left(\frac{\mathrm{N}_{n-k,1}(x_{n-k})}{x_{n-k-1}}\right)= \\
&=\mathrm{N}_{n-1,j-1}\left(\frac{\mathrm{N}_{n,1}(x_{n})}{x_{n-1}}\right)
\prod_{k=2}^{j}\mathrm{N}_{n-k,j-k}\left(\frac{\mathrm{N}_{n-k+1,1}(x_{n-k+1})}{x_{n-k}}\right).
\end{align*}
\allowdisplaybreaks[0]
The remaining part of the proof follows by Condition \textbf{(1)}.
\end{proof}
\begin{cor}\label{bound-odd}
Let $v(x)$ be a polynomial in $\mathrm{GF}(q,1)[x]$ and assume that \\  $f(x_{n - 1}, x_n) = x_n^2 + x_n - v(x_{n - 1})$ satisfies Conditions \textbf{(1)} and \textbf{(2)}, or Conditions \textbf{(1)} and \textbf{(2')}, and that $x_1$ and $\delta_1$ are not squares in $\mathrm{GF}(q,2)$. Then $x_n^2 \notin \mathrm{GF}(q,2^{n-1})$ and the order of $x_n$ is greater than
	$$2^{\frac{1}{2}(n^2 + 3n) + \mathrm{ord}_2(q - 1) - 2},$$
for all $n> 1$. The same lower bound also holds for the order of $\delta_n$ if $\delta_n^2 \notin \mathrm{GF}(q,2^{n - 1})$  for all $n > 1$.
\end{cor}
\begin{proof}
We know that $x_n\not\in \mathrm{GF}(q,2^{n-1})$ by Proposition \ref{general-odd}, therefore $x_n^2=-x_n+v(x_{n-1})\not\in \mathrm{GF}(q,2^{n-1})$ for all $n>1$. We show that the order of $x_n$ has a common factor with the odd number $\frac{q^{2^{n - j}} + 1}{2}$ proving that $x_n^{\frac{2(q^{2^n} - 1)}{q^{2^{n - j}} + 1}} \ne 1$, for $j = 1,2,\ldots,n-1$. 
For $j=1$, we have
\[
x_n^{\frac{2(q^{2^n} - 1)}{q^{2^{n - 1}} + 1}}=x_n^{2(q^{2^{n-1}}-1)} \ne 1,
\]
since $x_n^2\not\in \mathrm{GF}(q,2^{n - 1})$, as we have just seen. For $j\geq 2$, we get
\begin{align*}
x_n^{\frac{2(q^{2^n} - 1)}{q^{2^{n - j}} + 1}} &= \left(x_n^{\prod_{k = 1}^{j - 1}( q^{2^{n - k}}+1)}\right)^{2(q^{2^{n - j}} - 1)}= \mathrm{N}_{n,j - 1}(x_n)^{2(q^{2^{n - j}} - 1)}
\end{align*}
and the last member above is 1 only if $\mathrm{N}_{n,j - 1}(x_n)^2\in\mathrm{GF}(q,2^{n-j})$. We show that this is not possible. Consider $\mathrm{N}_{n,j }(x_n)=\mathrm{N}_{n-j+1,1 }(\mathrm{N}_{n,j-1 }(x_n))$. 
If $\mathrm{N}_{n,j - 1}(x_n)^2\in\mathrm{GF}(q,2^{n-j})$, then either $\mathrm{N}_{n,j }(x_n)=\mathrm{N}_{n,j - 1}(x_n)^2$ or $\mathrm{N}_{n,j }(x_n)=\mathrm{N}_{n,j - 1}(x_n)$. The latter equality is not possible since $\mathrm{N}_{n,j-1 }(x_n)$ is not a square in 
$\mathrm{GF}(q,2^{n-j+1})$ by Lemma \ref{lem:factor_odd} but 
$\mathrm{N}_{n,j }(x_n)\in \mathrm{GF}(q,2^{n-j})$ is a square in 
$\mathrm{GF}(q,2^{n-j+1})$. The former equality, by Lemma \ref{lem:factor_odd}, gives:
\allowdisplaybreaks
\begin{align*}
1&=\frac{x_{n - j} \prod_{k=1}^{j}\mathrm{N}_{n-k,j-k}\left(\frac{\mathrm{N}_{n-k+1,1}(x_{n-k+1})}{x_{n - k}}\right)}{x_{n - j+1}^2 \prod_{k=1}^{j-1}\left(\mathrm{N}_{n-k,j-k-1}\left(\frac{\mathrm{N}_{n-k+1,1}(x_{n-k+1})}{x_{n - k}}\right)\right)^2}=\\
&=\frac{x_{n - j}\frac{\mathrm{N}_{n-j+1,1}(x_{n-j+1})}{x_{n - j}} \prod_{k=1}^{j-1}\mathrm{N}_{n-k,j-k}\left(\frac{\mathrm{N}_{n-k+1,1}(x_{n-k+1})}{x_{n - k}}\right)}{x_{n - j+1}^2 \prod_{k=1}^{j-1}\left(\mathrm{N}_{n-k,j-k-1}\left(\frac{\mathrm{N}_{n-k+1,1}(x_{n-k+1})}{x_{n - k}}\right)\right)^2}=\\
&=\!\frac{\mathrm{N}_{n-j+1,1}(x_{n-j+1}) \!\prod_{k=1}^{j-1}\!\mathrm{N}_{n-j+1,1}\!\left(\!\mathrm{N}_{n-k,j-k-1}\!\left(\!\frac{\mathrm{N}_{n-k+1,1}(x_{n-k+1})}{x_{n - k}}\right)\!\right)}{x_{n - j+1}^2 \prod_{k=1}^{j-1}\left(\mathrm{N}_{n-k,j-k-1}\left(\frac{\mathrm{N}_{n-k+1,1}(x_{n-k+1})}{x_{n - k}}\right)\right)^2}\!=\\
&=\frac{(x_{n-j+1})^{q^{2^{n-j}}+1} \prod_{k=1}^{j-1}\left(\mathrm{N}_{n-k,j-k-1}\left(\frac{\mathrm{N}_{n-k+1,1}(x_{n-k+1})}{x_{n - k}}\right)\right)^{q^{2^{n-j}}+1}}{x_{n - j+1}^2 \prod_{k=1}^{j-1}\left(\mathrm{N}_{n-k,j-k-1}\left(\frac{\mathrm{N}_{n-k+1,1}(x_{n-k+1})}{x_{n - k}}\right)\right)^2}=\\
& =x_{n - j+1}^{q^{2^{n-j}}-1} \prod_{k=1}^{j-1}\left(\mathrm{N}_{n-k,j-k-1}\left(\frac{\mathrm{N}_{n-k+1,1}(x_{n-k+1})}{x_{n - k}}\right)\right)^{q^{2^{n-j}}-1}.
\end{align*}
\allowdisplaybreaks[0]
Since the last term is 1, then
$$x_{n-j+1} \prod_{k=1}^{j-1}\mathrm{N}_{n-k,j-k-1} \left(\frac{\mathrm{N}_{n-k+1,1}(x_{n-k+1})}{x_{n - k}}\right) \in\mathrm{GF}(q,2^{n-j}),$$
but this is impossible because $x_{n-j+1}$ is a non-square in $\mathrm{GF}(q,2^{n-j+1})$, by Proposition \ref{general-odd}, but 
$$\mathrm{N}_{n-k,j-k-1}\left(  \frac{\mathrm{N}_{n-k+1,1}(x_{n-k+1})}{x_{n - k}}\right)=\mathrm{N}_{n-k,j-k-1}\left(  \frac{f(x_{n-k},0)}{x_{n - k}}\right)$$ 
is a square in $\mathrm{GF}(q,2^{n-j+1})$, for each $k<j$, by Condition \textbf{(1)} and by multiplicativity of the norm.

This odd common factor ensures, by Lemma \ref{lem:estimate} with $a=q$, $b=n-j$ and $\ell=2$, the existence of a lower bound on the order of $x_n$, namely $p_j>2^{n - j + 1}$, for every $j = 1,2,\ldots,n-1$. Hence, the order is bounded below by $$2^{\frac{n(n + 1)}{2} - 1}=\prod_{j = 1}^{n - 1} 2^{n - j + 1} <\prod_{j = 1}^{n - 1}p_j .$$

The remaining term $2^{n + \mathrm{ord}_2(q - 1) - 1}$ follows as in \cite[Proposition 2]{burkhart2009finite}. 
By the repetition of the difference of squares formula, we get:
$$\mathrm{ord}_2 \left(\frac{q^{2^n} - 1}{2}\right) = \sum_{j = 0}^{n - 1} \mathrm{ord}_2(q^{2^j} + 1) + \mathrm{ord}_2(q - 1) - 1 = n + \mathrm{ord}_2(q - 1) - 1,$$
for all $n\geq 1$. It follows that $2^{n + \mathrm{ord}_2(q - 1) - 1}$ divides the order of $x_n$ because $x_n^{\frac{q^{2^n} - 1}{2}} = -1$ by Proposition \ref{general-odd}. The proof for $\delta_n$ is similar. 
\end{proof}

\section{Towers in even characteristic}\label{sec:even}

The even analogue of Conditions \textbf{(1)} and \textbf{(2)} in the odd case for polynomials $f(x,y)=y^3+y+v(x)$, with $v(x)\in \mathrm{GF}(2,1)[x]$, is:
\begin{description}
\item[(3)] There exists an integer $e\geq 0$ such that $f(x_{n-1},0)=x_{n-1}^{2^e}$ for all $n\geq 2$.
\end{description}
This means that we can restrict our study to polynomials in the form $f(x,y)=y^3+y+x^{2^e}$, with $e\geq 0$, and
deduce similar results as in the previous section. In Section \ref{sec:list-even}, we find some cases where the towers defined by  polynomials $f(x_{n - 1},x_n)$ are infinite and Galois. 
This is achieved by finding a suitable initial element $x_1\in \mathrm{GF}(2,6)$.
Under these hypotheses we have an analogue of Proposition \ref{general-odd}.
\begin{prop}\label{general-even}
Consider an infinite normal tower defined by $f(x_{n - 1},x_n) = x_n^3 + x_n + x_{n - 1}^{2^e}$ for a certain $e\geq 0$, for all $n>1$.
Let $p$ be a prime divisor of $|\mathrm{GF}(2,2\cdot 3^{n - 1})^*|$, for a suitable $n>1$, and assume that $x_{n - 1}$ is not a $p$-th power in the multiplicative group $\mathrm{GF}(2,2\cdot 3^{n - 1})^*$. Then $x_j$ is not a $p$-th power in the multiplicative group $\mathrm{GF}(2,2\cdot 3^{j})^*$, for $j \geq n$. 
\end{prop}
\begin{proof}
By assumption $f(x_{n-1},y)$ is irreducible, so $x_{n}\not\in\mathrm{GF}(2,2 \cdot 3^{n - 1})$ and $f(x_{n-1},y)$ is the minimum polynomial of $x_{n}$. We need to check that $x_{n}^{(4^{3^n}-1)/p}\not =1$. As in the proof of Proposition~\ref{general-odd}, we obtain:  
\begin{align*}
x_{n}^{(4^{3^n} - 1)/p} &= (x_{n}^{4^{2 \cdot 3^{n - 1}} + 4^{3^{n - 1}} + 1})^{(4^{3^{n - 1}} - 1)/p}= \\
&= \mathrm{N}(x_n)^{(4^{3^{n - 1}} - 1)/p} = f(x_{n - 1},0)^{(4^{3^{n - 1}} - 1)/p},
\end{align*}
where $\mathrm{N}(x_n) = x_{n}^{4^{2 \cdot 3^{n - 1}} + 4^{3^{n - 1}} + 1} = f(x_{n - 1},0)$ is the norm of $x_n$ over $\mathrm{GF}(2,2\cdot 3^{n-1})$. The last term is not equal to $1$ because $x_{n - 1}$ is not a $p$-th power in $\mathrm{GF}(2,2 \cdot 3^{n - 1})$, hence, by Condition \textbf{(3)}, 
$f(x_{n - 1},0)$ is not a $p$-th power as well.
\end{proof}
The analogue of Lemma \ref{lem:factor_odd} in even characteristic is the following: 
\begin{lem}\label{lem:factor_even} Let $e\geq 0$, $n\geq 2$ and $j<n$ be positive integers, then
$$\frac{\mathrm{N}_{n,j}(x_n)}{x_{n - j}^{2^{ej}}} = \prod_{k = 1}^{j} \mathrm{N}_{n - k,j-k} \left(\frac{\mathrm{N}_{n - k+1,1}(x_{n - k+1})}{x^{2^e}_{n - k }}\right)^{2^{e(k-1)}}.$$
In particular, if the tower defined by $f(x_{n-1},x_n)$ satisfies Condition \textbf{(3)} for a certain $e\geq 0$, then 
$\mathrm{N}_{n,j}(x_n)=x_{n - j}^{2^{ej}}.$
\end{lem}
\begin{proof} By induction on $j$. For $j=1$ the result is trivial. Let $j\geq 2$ and assume the result holds for $j-1$, then:
\begin{align*}
&\frac{\mathrm{N}_{n,j}(x_n)}{x^{2^{ej}}_{n-j}}\!=\! \frac{x_n^{(4^{2\cdot 3^{n-1}}+4^{3^{n-1}}+1)\prod_{i=2}^{j}(4^{2\cdot 3^{n-i}}+4^{3^{n-i}}+1)}}{x^{2^{ej}}_{n-j}}\!= \\
&=\!\left(\!\frac{x_n^{4^{2\cdot 3^{n-1}}+4^{3^{n-1}}+1}}{x_{n-1}^{2^e}}\!\right)^{\prod_{i=2}^{j}(4^{2\cdot 3^{n-i}}+4^{3^{n-i}}+1)}\!\!\!\left(\frac{x_{n-1}^{\prod_{i=2}^{j} (4^{2\cdot 3^{n-i}}+4^{3^{n-i}}+1)}}{x^{2^{e(j-1)}}_{n-j}}\right)^{2^e}\!\!\!= \\
&=\!\left(\frac{\mathrm{N}_{n,1}(x_n)}{x_{n-1}^{2^e}}\right)^{\prod_{i=2}^{j}(4^{2\cdot 3^{n-i}}+4^{3^{n-i}}+1)}\left(\frac{\mathrm{N}_{n-1,j-1}(x_{n-1})}{x_{n-j}^{2^{e(j-1)}}}\right)^{2^e}= \\
&=\!\mathrm{N}_{n-1,j-1}\left(\frac{\mathrm{N}_{n,1}(x_{n})}{x_{n-1}^{2^e}}\right)
\prod_{k=1}^{j-1}\mathrm{N}_{n-k-1,j-k-1}\left(\frac{\mathrm{N}_{n-k,1}(x_{n-k})}{x_{n-k-1}^{2^e}}\right)^{2^{ek}}= \\
&=\!\mathrm{N}_{n-1,j-1}\left(\frac{\mathrm{N}_{n,1}(x_{n})}{x_{n-1}^{2^e}}\right)\prod_{k=2}^{j}\mathrm{N}_{n-k,j-k}\left(\frac{\mathrm{N}_{n-k+1,1}(x_{n-k+1})}{x_{n-k}^{2^e}}\right)^{2^{e(k-1)}}.
\end{align*}
The remaining part of the proof follows by Condition \textbf{(3)}. \end{proof}
\begin{cor}\label{bound-even}
Consider an infinite normal tower defined by
$f(x_{n-1},x_n)=x^3_n+x_n+x_{n-1}^{2^e}$, 
for a certain $e\geq 0$, for all $n>1$. If $x_1$ is not a cube in $\mathrm{GF}(2,6)$, then $x_n^3 \notin \mathrm{GF}(2,2 \cdot 3^{n - 1})$ for all $n\geq 2$ and the order of $x_n$ in the tower defined by $f(x_{n - 1},x_n)$ is greater than $$3^{\frac{1}{2}(n^2 + 3n)-1}.$$
\end{cor}
\begin{proof}
The proof is similar to the proofs of Corollary \ref{bound-odd} and \cite[Proposition 4]{burkhart2009finite}. 
We know that $x_n\not\in \mathrm{GF}(2,2\cdot 3^{n - 1})$ by
Proposition~\ref{general-even}, therefore $x^3_n=x_n+v(x_{n-1})\not \in \mathrm{GF}(2,2\cdot 3^{n - 1})$.
It follows that $\left(x^3_n\right)^{2^e}$ does not belong to $\mathrm{GF}(2,2\cdot 3^{n - 1})$.
In order to show that the order of $x_n$ has a common factor with $\textstyle \frac{1}{3}(4^{2 \cdot 3^{n - j}} + 4^{3^{n - j}} + 1)$, we show that $x_n^\frac{3(4^{3^n} - 1)}{4^{2 \cdot 3^{n - j}} + 4^{3^{n - j}} + 1} \ne 1$, for $j = 1,2,\ldots,n-1$. 
We have:
\begin{align*}
x_n^\frac{3(4^{3^n} - 1)}{4^{2 \cdot 3^{n - j}} + 4^{3^{n - j}} + 1} &= x_n^{\frac{4^{3^n} - 1}{4^{3^{n - j}} - 1} \cdot \frac{3(4^{3^{n - j}} - 1)}{4^{2 \cdot 3^{n - j}} + 4^{3^{n - j}} + 1}}= x_n^{\frac{3(4^{3^n} - 1)(4^{3^{n - j}} - 1)}{4^{3^{n - j + 1}} - 1}} = \\
&=x_n^{3(4^{3^{n - j}} - 1)\prod_{i = 1}^{j - 1} (4^{2 \cdot 3^{n - i}} + 4^{3^{n - i}} + 1)}=
\mathrm{N}_{n,j - 1}(x_n)^{3(4^{3^{n - j} - 1})}.
\end{align*}
By Lemma \ref{lem:factor_even} 
we have that $\mathrm{N}_{n,j}(x_n)=x_{n-j}^{2^{ej}}$, for $j = 1,2,\ldots,n-1$. But $\left(x_{n - j + 1}^{2^{e(j-1)}}\right)^3$ does not belong to $\mathrm{GF}(2,2 \cdot 3^{n - j})$ for all $j \geq 1$. 
It follows that $\mathrm{N}_{n,j - 1}(x_n)^{3(4^{3^{n - j} - 1})}$ cannot be equal to $1$. This ensures, by Lemma \ref{lem:estimate} with $a=4$, $b=n-j$ and $\ell=3$, the existence of a lower bound on the order of $x_n$, namely $p_j>3^{n - j + 1}$, for every $j = 1,2,\ldots,n-1$. Hence, we get a lower bound for the order of $x_n$, which is $3^{\frac{n(n + 1)}{2}-1}=\prod_{j = 1}^{n - 1} 3^{n - j + 1}<\prod_{j = 1}^{n - 1}p_j$.
	
The remaining term $3^{n}$ follows by the computation of the power of 3 dividing the order of $x_n$. By the repetition of the difference of cubes formula, we have:
$$\mathrm{ord}_3 \left(\frac{4^{3^n} - 1}{3}\right) = \sum_{j = 0}^{n - 1} \mathrm{ord}_3(4^{2 \cdot 3^j} + 4^{3^j} + 1) + \mathrm{ord}_3(4 - 1) - 1 = n,$$
for all $n \geq 1$. This term divides the order of $x_n$, since $x_n^{\frac{4^{3^n} - 1}{3}} \neq 1$, by Proposition~\ref{general-even}.
\end{proof}

\section{Examples of good towers in odd characteristic}\label{sec:list-odd}

In this section we find high order elements in $\mathrm{GF}(q,2^n)$, for odd $q$, using five good towers. In this section, we denote by $\varepsilon$ the element $4^{-1}$ inside $\mathrm{GF}(q,1)$. We consider the polynomials $f_i(x_{n - 1}, x_n) := x_n^2 + x_n - v_i(x_{n - 1})$, for $i \in\{1, 2,\ldots, 5\}$,
where $v_i(x)$ is a polynomial chosen as follows:
\begin{enumerate}
    \item $v_1(x) := \varepsilon x;$
    \item $v_2(x) := 4x{(x + 3\varepsilon)}^2;$
    \item $v_3(x) := 2\varepsilon x;$
    \item $v_4(x) := 8x{(2x + 3\varepsilon)}^2;$
    \item $v_5(x) := 8x{(x + 3\varepsilon)}^2.$
\end{enumerate}
\begin{rem}\label{rem:conditions}
Condition \textbf{(1)} holds for all the previous polynomials and the relation between two consecutive discriminants is given respectively by:
\begin{align*}
g_1(\delta_{n - 1}, \delta_n) &= \delta_n^2 - \delta_n- \varepsilon \delta_{n - 1} + \varepsilon; \\
g_2(\delta_{n - 1}, \delta_n) &= \delta_n^2 - \delta_n -4 \delta_{n - 1}^3 + 6\delta_{n - 1}^2 - 9\varepsilon\delta_{n - 1} + \varepsilon; \\
g_3(\delta_{n - 1}, \delta_n) &= \delta_n^2 - \delta_{n - 1};
\\g_4(\delta_{n - 1}, \delta_n) &= \delta_n^2 + 48\delta_{n - 1}\delta_n - 256\delta_{n - 1}^3 + 288\delta_{n-1}^2 - 81\delta_{n - 1}; \\
g_5(\delta_{n-1}, \delta_n) &= \delta_n^2 - 16\delta_{n - 1}^3 + 24\delta_{n - 1}^2 - 9\delta_{n - 1}.
\end{align*}
The first two towers satisfy Condition \textbf{(2)}. 
In fact
\begin{align*}
g_1(\delta_{n-1},0)&=-\varepsilon (1+4x_{n-1})+\varepsilon=-x_{n-1};\\
g_2(\delta_{n-1},0)&=x_{n-1}(x_{n-1}+3\varepsilon)^2(x_{n-1}^3+6x_{n-1}^2+9\varepsilon^2 x_{n-1}+3\varepsilon^3)^2.
\end{align*}
Similarly the last three towers satisfy Condition \textbf{(2')}. In fact,
\begin{align*}
g_3(\delta_n, 0) &= -\delta_{n};\\
g_4(\delta_n, 0) &= -256\delta_n(\delta_n - 9 \varepsilon^2)^2; \\
g_5(\delta_n, 0) &= -16\delta_n(\delta_n - 3\varepsilon)^2.
\end{align*}
Hence, Proposition \ref{general-odd} applies to $f_i(x_{n - 1}, x_n)$, for $i \in\{ 1, 2, \ldots, 5\}$ once we have some starting points. 
\end{rem}
The next two lemmas ensures the existence of a non-square $x_1$ such that $\delta_1$ is a non-square in $\mathrm{GF}(q,2)$ as well. This would be the corresponding analogue of \cite[Lemma 3]{burkhart2009finite}, but 
here we also need that \textit{both} $x_1$ and $\delta_1$ must be non-squares. This requires more effort, especially for the last tower $f_5(x_{n-1},x_n)$ below, but, 
as a balance, this gives 
a lower bound for the order of $x_n$, also. 

The present proof relies mainly on elementary combinatorial arguments. 
\begin{lem}\label{lem:nonsqnear}
Let $c\in \mathrm{GF}(q,1)$ be a non-zero element. There is at least a non-square $x_1 \in \mathrm{GF}(q,2)$ such that $x_1 + c$ is a non-square as well.
\end{lem}
\begin{proof}
Consider the action $\rho$ of $\mathrm{GF}(q,1)$ on $\mathrm{GF}(q,2)$ as an additive group, namely $\rho_g(x) = x + g$, for $g \in \mathrm{GF}(q,1)$ and $x \in \mathrm{GF}(q,2)$. Then, $\mathrm{GF}(q,2)$ is partitioned into $q$ orbits. There are exactly $\textstyle\frac{1}{2}(q^2 + 1)$ squares in $\mathrm{GF}(q,2)$. Among these, there are all the elements of the orbit $\mathrm{GF}(q,1)$. It follows that there are exactly $\textstyle\frac{1}{2}(q^2 - 2q + 1)$ square elements in $q - 1$ orbits. Hence, there is at least one orbit with at most $\textstyle\frac{1}{2}(q - 1)$ square elements and at least $\textstyle\frac{1}{2}(q + 1)$ non-square elements. We denote this orbit by $S$. It follows that there are at least two consecutive non-squares in $S$ under the repeated action of $\rho_c$, namely $a$ and $\rho_c(a) = a + c$. The lemma follows by choosing $x_1 = a$.
\end{proof}
\begin{exa}\label{exa1}
Consider, $q = 3$ and $c=1$. Denote by $z$ a generator of $\mathrm{GF}(3,2)^*$ satisfying $z^2 = z + 1$. There are exactly $5$ squares in $\mathrm{GF}(3,2)^*$, but $3$ of them are in the same orbit $\mathrm{GF}(3,1)$. The remaining ones are $z^2 = z + 1$ and $z^6 = 2z + 2$. One can check that they belong to the orbits $S_1 = (z; \quad z + 1 = z^2; \quad z + 2 = z^7)$ and $S_2 = (2z = z^5; \quad 2z + 1 = z^3; \quad 2z + 2 = z^6)$. As $x_1$ we can choose the element $2z$ or $z + 2$. They are both roots of the polynomial $x^2 = 2x + 1$, so we use this polynomial for $q = 3$ in Table \ref{tab:res1} in Section~\ref{sec:numres}.
\end{exa}

In order to show the existence of a suitable initial element $x_1$ for the tower defined by $f_5(x_{n-1},x_n)$ we prove the following lemma. 
\begin{lem}\label{lem:nonsqnear5}Let $q$ be an odd prime and let $p(x)$ be a cubic polynomial in $\mathrm{GF}(q,1)[x]$ without multiple roots, such that $p(0)\not = 0$. Then:
\begin{enumerate}[i)]\item\label{nonsqnear5_C1} The curve $C_1: y^2 = p(x)$ has at most $q^2 + 2q$ affine $\mathrm{GF}(q,2)$-rational points and the curve $C_2: y^2 = p(x^2)$ has at least $q^2 - 4q - 1$ affine $\mathrm{GF}(q,2)$-rational points.\item\label{nonsqnear5_exists} If $q\geq 11$, then there is at least a non-square $x_1\in\mathrm{GF}(q,2)$ such that $p(x_1)$ is a non-square in $\mathrm{GF}(q,2)$ as well.
\end{enumerate}
\end{lem}
\begin{proof} \ref{nonsqnear5_C1}) We observe that $p(x^2)$ is square-free since $p(x)$ is square-free and $p(0)\ne 0$ by hypothesis. The first statement follows by Weil bound $|N-(q^2 + 1)| \leq 2gq$, for every smooth projective curve of genus $g$ with $N$ points over $\mathrm{GF}(q,2)$, since $C_1$ is an elliptic curve and $C_2$ has genus at most 2, see \cite[Propositions 6.1.3 (a) and 6.2.3 (b)]{stichtenoth2009functionfields}. It is well known that the number of points at infinity is $1$ in an elliptic curve and it is at most $2$ in a genus $2$ curve. Hence, \ref{nonsqnear5_C1}) is proved.

\ref{nonsqnear5_exists}) By contradiction we assume that $p(\alpha)$ is a square for all non-square $\alpha \in \mathrm{GF}(q,2)$. Let $\beta\in \mathrm{GF}(q,2)$ be a square root of $p(\alpha)$. Since there are exactly $\textstyle\frac{1}{2}(q^2 - 1)$ non-squares in $\mathrm{GF}(q,2)$ and $\beta \neq 0$, except at most for $3$ choices of $\alpha$, then the pairs $(\alpha,\beta)$ and $(\alpha,-\beta)$ produce at least $q^2 - 4$ distinct points of $C_1$. We show that such points are too many. We estimate the number of squares $\alpha$ such that $p(\alpha)$ is also a square in $\mathrm{GF}(q,2)$. Each point $(t,y)$ in $C_2$ corresponds to the point $(x,y)$ in $C_1$ with $x = t^2$. This correspondence is not $1-1$ because, when $t\ne 0$, the point $(-t,y)$ determines the same point in $C_1$. Let $N$ be the number of affine $\mathrm{GF}(q,2)$-rational points of $C_2$, then $C_1$ must have more than $\textstyle\frac{N}{2}$ affine $\mathrm{GF}(q,2)$-rational points $(x,y)$ with $x$ being a square in $\mathrm{GF}(q,2)$. By Part \ref{nonsqnear5_C1}), we have $N\ge q^2 - 4q -1$. Counting the points of $C_1$ we get, again by Part \ref{nonsqnear5_C1}), $\textstyle q^2 - 4 +\frac{1}{2}(q^2 - 4q -1) \le q^2 + 2q$ which yields, after a straightforward computation, $q^2 - 8q - 9 \le 0$. It follows that $q \le 9$, which is contrary to our assumption on $q$. Hence, there is at least one non-square $x_1 \in \mathrm{GF}(q,2)$ such that $p(x_1)$ is a non-square too.
\end{proof}
\begin{rem}\label{rem5}For suitable polynomials $p(x)$, Part \ref{nonsqnear5_exists}) of the previous lemma also holds for odd primes $q < 11$. In the sequel, we are interested in \begin{align*}p(x)\! &= \!1 + 4v_5(x) \!=\! 1 + 32x(x + 3\varepsilon)^2\!=\! 32\!\left(x + \frac{1}{2}\right)\!\!\left(x^2 + x + \frac{1}{16}\right)\!= \\ &= 32\left(x +\frac{1}{2}\right)\left(x + \frac{1}{2} - a\right)\left(x + \frac{1}{2} + a\right),
\end{align*} 
where $a$ is a square root of $\textstyle\frac{3}{16}$ in $\mathrm{GF}(q,2)$. We are interested in  this polynomial because $\delta_1=p(x_1)$, for $f_5$, so we need that both $x_1$ and $\delta_1$ are non-squares. It follows that $p(x)$ is square-free for $q = 5$ and $q = 7$. For $q = 5$, if we choose $x_1$ being a root of the polynomial $x^2 + 4x + 2$, then $p(x_1) = x_1^5$. Hence, both $x_1$ and $p(x_1)$ are non-square in $\mathrm{GF}(5,2)$. 
Similarly, for $q = 7$, if we choose $x_1$ as a root of $x^2 + 5x + 5$, then both $x_1$ and $p(x_1)$ are non-squares in $\mathrm{GF}(7,2)$. 
Finally, for $q = 3$, we have that $p(x)$ has multiple roots, but Part \ref{nonsqnear5_exists}) of Lemma \ref{lem:nonsqnear5} still holds. In fact, if we choose $x_1$ as a root of the polynomial $x^2 + 2x + 2$, as in Example \ref{exa1}, then $p(x_1) = x_1$, hence $p(x_1)$ is a non-square as well. Kindly refer to Remark \ref{rem_f4=f5} for further explanations. We use the aforementioned examples in Table \ref{tab:res5} in Section~\ref{sec:numres}.
\end{rem}
The following corollary ensures the existence of towers defined by $f_i(x_{n - 1},x_n)$ generating high order elements for $i \in\{ 1,2,\ldots,5\}$.
\begin{cor}\label{bound1-4}
The polynomials $f_i(x_{n - 1},x_n)$, for $i \in\{1,2,\ldots,5\}$, define infinite towers of fields. Moreover, for a suitable choice of $x_1$, the order of $x_n$ in $\mathrm{GF}(q,2^n)$ is greater than $2^{\frac{1}{2}(n^2 + 3n) + \mathrm{ord}_2(q - 1) - 2}$. The same bound holds for $\delta_n$ in the towers defined by $f_1(x_{n - 1},x_n)$ and $f_2(x_{n - 1},x_n)$ and, when $q > 3$, for $\delta_n$ in the tower defined by $f_4(x_{n - 1},x_n)$.
\end{cor}
\begin{proof}
First, for each tower considered, we show the existence of a non-square starting point $x_1$ such that the discriminant $\delta_1$ is a non-square as well. A straightforward computation shows that $\delta_1 = x_1 + 1$ for $f_1$ and that $\delta_1 = 16x_1^3 + 24x_1^2 + 9x_1 + 1 = (x_1 + 1)(4x_1 + 1)^2$ for $f_2$. Hence, for the first two polynomials, it is enough to choose $x_1$ as in Lemma \ref{lem:nonsqnear} with $c=1$. A straightforward computation also shows that $\textstyle \delta_1 = 2\left(x_1 + \frac{1}{2}\right)$ for $f_3$ and that $\textstyle \delta_1 = 128\left(x_1 + \frac{1}{2}\right)(x_1 + 2\varepsilon^2)^2$ for $f_4$. Hence, for the third and the fourth polynomial, it is enough to choose $x_1$ as in Lemma \ref{lem:nonsqnear} with $\textstyle c=\frac{1}{2}$. For the last tower, by Remark~\ref{rem5}, we can take $x_1$ as in Remark \ref{rem5} for $q \leq 7$ and we can take $x_1$ as in Lemma~\ref{lem:nonsqnear5} for $q \geq 11$.

Now, we know, by Remark \ref{rem:conditions}, that all the considered towers satisfy Conditions \textbf{(1)} and \textbf{(2)}, or Conditions \textbf{(1)} and \textbf{(2')}. Therefore, the result for $x_n$ follows by Corollary \ref{bound-odd}. For $\delta_n$ we have to check that $\delta_n^2 \notin \mathrm{GF}(q,2^{n-1})$ for $n > 1$, in the tower defined by $f_1(x_{n-1},x_n)$ and $f_2(x_{n-1},x_n)$, for $q\geq 3$, and by $f_4(x_{n-1},x_n)$ for $q > 3$. But this follows by the expression of $g_1(\delta_{n - 1},\delta_{n })$,
 $g_2(\delta_{n - 1},\delta_{n })$ and $g_4(\delta_{n - 1},\delta_{n })$ in Remark~\ref{rem:conditions}.
\end{proof}
As in \cite{burkhart2009finite}, the bound of the previous corollary does not seem to be sharp, in fact in many cases we were able to construct generators of the multiplicative group $\mathrm{GF}(q,2^n)^*$, whose order is $q^{2^n} - 1$, which is much higher than $2^{\frac{n^2}{2}}$. The interested reader can compare the tables in Section \ref{sec:numres} with the experimental results of \cite{burkhart2009finite}.
\begin{rem}\label{rem_f4=f5} 
The bound in the Corollary \ref{bound1-4} above, does not hold for $\delta_n$ in the tower defined by $f_3(x_{n - 1}, x_n)$ and $f_5(x_{n - 1}, x_n)$. 
In fact, $\delta_n^2 \in \mathrm{GF}(q,2^{n - 1})$, for all $n > 1$, which can be verified easily. The interested reader can see the numerical results of Tables \ref{tab:res3} and \ref{tab:res5} in Section \ref{sec:numres}.
A careful comparison between these two tables reveals an interesting difference when $q > 3$. In fact, the order of the discriminant $\delta_n$ turns out to grow very slowly in Table \ref{tab:res3} in comparison to Table \ref{tab:res5}. The reason is that in the former tower the discriminants satisfy the relation
$g_3(\delta_{n - 1}, \delta_n) = \delta_n^2 - \delta_{n - 1} = 0,$
which yields $\delta_n^{2^{n - 1}} = \delta_{n - 1}^{2^{n - 2}} = \ldots = \delta_{1} \in \mathrm{GF}(q,2).$  This implies that we can estimate the order of $\delta_n$, which turns out to be lower than $2^{n - 1 + \mathrm{ord}_2(q^2 - 1)}$. In the tower defined by $f_5(x_{n - 1},x_n)$, we have that $\delta_n^{2^j} \in \mathrm{GF}(q,2^{n - j})$ holds for $j = 1$, but not for all $j < n$. This explains why the order grows comparatively faster when $q > 3$. In the case $q = 3$ the polynomial equation $g_5(\delta_{n - 1},\delta_n) = \delta_n^2 - \delta_{n - 1}^3 = 0$ gives $\delta_n^{2^{n - 1}} = \delta_{1}^{3^{n - 1}} \in \mathrm{GF}(3,2)$. This explains why the numerical results for the order of $\delta_n$ are similar to the tower defined by $f_3(x_{n - 1},x_n)$.
\end{rem}
\begin{rem}
From the relation $g_4(\delta_{n - 1},\delta_n)=0$ between $\delta_n$ and $\delta_{n - 1}$ in the fourth tower, for $q = 3$, we get $g_4(\delta_{n - 1}, \delta_n) = \delta_n^2 - \delta_{n - 1}^3 = 0$. Hence, we observe that the proof of last corollary does not work when $q = 3$. We also point out that $f_4(x_{n - 1}, x_n) = f_5(x_{n - 1}, x_n)$ when $q = 3$. This fact explains why the numerical results in Tables \ref{tab:res4} and \ref{tab:res5} have the same values in the first two columns.
\end{rem}
Of course could exist other towers satisfying analogues of Conditions \textbf{(1)} and \textbf{(2)} or Conditions \textbf{(1)} and \textbf{(2')} above. 
An extensive computer search could show the non-existence of similar examples of the form $f(x_{n - 1},x_n) = x_n^2 + x_n + v(x_{n - 1})$, with $\deg(v(x)) \leq 3$, at least for small prime fields.

\section{Examples of good towers in even characteristic}\label{sec:list-even}

In this section we 
list polynomials generating high order elements, as in
Section \ref{sec:list-odd}.
We have to adapt some proofs in even characteristic, since we have to prove that our cubic polynomials $f(x_{n-1},y)$ are Galois and irreducible in $\mathrm{GF}(2,2\cdot 3^{n-1})[y]$. Let $e$ be a non-negative integer. In the following results, we prove that $f(x_{n - 1},x_n):= x_n^3+x_n+x_{n-1}^{2^e}$ actually defines an infinite normal separable tower.

\begin{lem}\label{lem:cubicroots4}
Let $e$ and $n$ be integers such that $e\geq 0$ and $n\ge 2$, and let $x_{n-1}\in \mathrm{GF}(2,2\cdot 3^{n-1})$. Assume that $u_n^3\in \mathrm{GF}(2,2\cdot 3^{n-1})$ is a root of the quadratic polynomial $y^2+x_{n-1}^{2^e}y+1$ and that $x_n:=u_n+u^{-1}_n\notin \mathrm{GF}(2,2\cdot 3^{n-1})$ is a root of the cubic polynomial $y^3+y+x_{n-1}^{2^e}$. 
Let $u_{n+1}\in \mathrm{GF}(2,2\cdot 3^{n+1})$ be a third root of $u_n^{2^e}$. Then:
\begin{enumerate}[i)] 
\item\label{item_cond1}
 $u_{n+1}\notin \mathrm{GF}(2,2\cdot 3^{n})$;

\item\label{item_cond2}
$u_{n+1}^3$ and $u_{n+1}^{-3}$ are the roots of $y^2+x_n^{2^e}y+1$;

\item\label{item_cond3}
 $x_{n+1}:=u_{n+1}+u_{n+1}^{-1}$ is a root of $y^3+y+x_n^{2^e}$ and $x_{n+1}\notin \mathrm{GF}(2,2\cdot 3^{n})$.
\end{enumerate} 
\end{lem}
\begin{proof}
Part \ref{item_cond1}) follows since $u_{n+1}^9=(u_{n+1}^3)^3=(u_n^{2^e})^3=(u_n^3)^{2^e}$ belongs to $\mathrm{GF}(2,2\cdot 3^{n-1})$ and since $\mathrm{GF}(2,2\cdot 3^{n})$ does not contain any 9-th root of non-cubic elements in $\mathrm{GF}(2,2\cdot 3^{n-1})$ because 9 does not divide
$$\frac{|\mathrm{GF}(2,2\cdot 3^{n})^*|}{|\mathrm{GF}(2,2\cdot 3^{n-1})^*|}=1+4^{3^{n-1}}+4^{2\cdot 3^{n-1}},$$ for all $n\ge 1$.

Part \ref{item_cond2}) follows by straightforward verification.

The last part follows by Lemma \ref{lem:cubicroots2} and by Parts \ref{item_cond1}) and \ref{item_cond2}).
\end{proof}

Part iii) in the previous Lemma shows 
by induction that if $f(x_1,y)=y^3+y+x_{1}^{2^e}$
is irreducible over  
$\mathrm{GF}(2,6)$, then 
$f(x_n,y)=y^3+y+x_{n}^{2^e}$ is also irreducible over  
$\mathrm{GF}(2,2\cdot 3^{n})$ for all $n>1$.
It follows that the Galois groups of the splitting field of $f(x_n,y)$ is either the cyclic group
$\mathbb{Z}/3 \mathbb{Z}$ or the symmetric group $\mathrm{S}_3$, at each iteration. 
In the same spirit of the results above, we show that if the first polynomial $f(x_1,y)$ is normal (i.e., the Galois group of the splitting field is
$\mathbb{Z}/3 \mathbb{Z}$) then 
$f(x_n,y)$ is normal at each iteration.

\begin{lem}\label{lem:galois}
Let $e\geq 0$ and $x_{n - 1} \in \mathrm{GF}(2,2 \cdot 3^{n - 1})$. Assume that $f(x_{n - 1},y)=y^3+y+x_{n-1}^{2^e}$ splits into linear factors in $\mathrm{GF}(2,2 \cdot 3^{n})[y]$. Then $f(x_{n},y)$ splits into linear factors in $\mathrm{GF}(2,2 \cdot 3^{n + 1})[y]$.
\end{lem}
\begin{proof}
 Let $r_1, r_2, r_3 \in \mathrm{GF}(2,2 \cdot 3^{n})$ be the roots of $f(x_{n - 1},y)$ and choose $x_n = r_1$. Let $R(x_n,y) = y^2 + x_n^{2^e} y + x_n^{2^{e+1}} + 1$ be the quadratic resolvent of $f(x_{n},y)$. 
Applying Frobenius automorphism, we know that the roots satisfy $r_2^{2^e} + r_3^{2^e} = r_1^{2^e}$ and $r_1^{2^e}r_2^{2^e}r_3^{2^e}=(x_{n-1}^{2^e})^{2^e}$. Hence, we get
\begin{align*}
R(x_n,y)&=y^2 + r_1^{2^e}y + r_1^{2^{e+1}} + 1=y^2 + r_1^{2^e}y + \frac{(r_1^3 + r_1)^{2^e}}{r_1^{2^e}}= \\
&= y^2 + r_1^{2^e}y + \frac{x_{n-1}^{2^{2e}}}{r_1^{2^e}}=(y + r_2^{2^e})(y + r_3^{2^e}).
\end{align*}
It follows that $R(x_n,y)$ splits in $\mathrm{GF}(2,2 \cdot 3^{n})[y]$. Therefore, $f(x_{n},x_{n + 1})$ splits into linear factors in $\mathrm{GF}(2,2 \cdot 3^{n + 1})[y]$ by Lemma \ref{lem:quadratic_resolvent}.
\end{proof}

We summarize the results above in the following Corollary,
which provides a good initial choice for $x_1$, resulting in
$f(x_{n-1},x_n)$ to be a normal separable recursive tower.

\begin{cor}\label{cor:eventowerorder}
Let $e\geq 0$ be an integer. Then $f(x_{n-1},x_n):= x_n^3+x_n+x_{n-1}^{2^e}$ defines an infinite tower of fields and, for a suitable choice of $x_1$, the order of $x_n\in \mathrm{GF}(2,2\cdot 3^n)$, for $n \ge 2 $, is greater than $3^{\frac{1}{2}(n^2 + 3n)-1}.$
\end{cor}
\begin{proof}
Let $x_1$ be one of the roots of $h(x) := x^6 + x^5 + x^3 + x^2 + 1.$ The reader can verify that each root of this polynomial is not a cube in $\mathrm{GF}(2,18)$. Moreover, the quadratic resolvent $R(x_1,y)=y^2+x_1y+x_1+1$ of $f(x_{1},y)= y^3+y+x_{1}$ is reducible in $\mathrm{GF}(2,6)[y]$ and the roots of $y^2+x_1y+1$ are not cubes in $\mathrm{GF}(2,6)$. Applying the Frobenius automorphism, this implies that, for all $e\ge 0$, the roots of $y^2+x_1^{2^e}y+1$ are not cubes and the quadratic resolvent $R(x_1^{2^e},y)=y^2+x_1^{2^e}y+x_1^{2^{e+1}}+1$ of $f(x_{1},y)= y^3+y+x_{1}^{2^e}$ is reducible in $\mathrm{GF}(2,6)[y]$.

By Lemma~\ref{lem:cubicroots4}, Part \ref{item_cond3}), the fact that the roots of $y^2+x_1^{2^e}y+1$ are not cubes implies that $f(x_{n},y)= y^3+y+x_{n}^{2^e}$ is irreducible for each $n\ge 1$. Hence $f(x_{n-1},x_n)$ defines an infinite tower of fields. By Lemma~\ref{lem:galois}, the condition on the resolvent implies that this tower is Galois.

Since $f$ clearly satisfies Condition \textbf{(3)} of Section \ref{sec:even}, so the proof follows by Corollary~\ref{bound-even}.
\end{proof}
In Table \ref{tab:res6} of Section~\ref{sec:numres} we collated the numerical results for $f_6(x_{n - 1},x_n) := x_n^3 + x_n + x_{n - 1}$ and $f_7(x_{n - 1},x_n) := x_n^3 + x_n + x_{n - 1}^2$ corresponding to $e=0$ and $e=1$, respectively.
The initial element $x_1$ is one of the roots of $h(x) := x^6 + x^5 + x^3 + x^2 + 1$ as explained in the proof of Corollary \ref{cor:eventowerorder}.

\section{Numerical results}\label{sec:numres}

In this section, we have collated the multiplicative orders $o(x_n)$ (and $o(\delta_n )$ for $q$ odd) for small $n$ in the towers defined by $f_i (x_{n - 1}, x_n)$, for $i = 1, 2,\ldots,7$. In most of the cases we obtained generators of the multiplicative groups $\mathrm{GF}(q, 2^n)^*$ and $\mathrm{GF}(2, 2 \cdot 3^n)^*$. We tabulated base 2 logarithm of the orders as they grow    
exponentially. The interested reader can also find the lower and upper bounds for $o(x_n)$ and $o(\delta_n)$ listed in Tables \ref{tab:bounds} and \ref{tab:res6}, for odd and even characteristic respectively. Finally, in Table \ref{tab:ComparativeAnalysis}, we compare one of our example in characteristic $3$, with the constructions of \cite{burkhart2009finite} and \cite{cohen1992explicit}.

MAGMA \cite{MAGMA} computational algebra system was used for the experiments and a sample MAGMA code and output, for $q = 11$ can be found in \cite{CODE}. The performance of the code depends on the efficiency of the \textit{root} finding algorithm that one uses. We have used the standard function of MAGMA \cite{MAGMA} for finding roots.

\begingroup
\tiny{
\begin{table}[ht]
\caption{Results for $f_1(x_{n - 1},x_n)$ for odd $q\leq 11$.}
\begin{center}
\begin{adjustwidth}{-0.35in}{-0.35in}
\begin{tabular}{|c||c|c||c|c||c|c||c|c|}
\hline              

$q$ & 3 & 3 & 5 & 5 & 7 & 7 & 11 & 11 \\
\hline

$x_1^2=$ & $2x_1 + 1$ & $2x_1 + 1$  & $3x_1 + 2$ & $3x_1 + 2$ & $x_1 + 4$ & $x_1 + 4$ & $4x_1 + 9$ & $4x_1 + 9$ \\
\hline
\hline
$n$ & $\log_2$(o($x_n$)) & $\log_2$(o($\delta_n$)) & $\log_2$(o($x_n$)) & $\log_2$(o($\delta_n$)) & $\log_2$(o($x_n$)) & $\log_2$(o($\delta_n$)) & $\log_2$(o($x_n$)) & $\log_2$(o($\delta_n$)) \\ 
\hline  

$1$ & 3.0 & 3.0 & 4.6 & 3.0 & 5.6 & 5.6 & 6.9 & 5.3 \\
\hline

$2$ & 6.3 & 6.3 & 9.3 & 9.3 & 11.2 & 11.2 & 13.8 & 13.8 \\
\hline

$3$ & 12.7 & 12.7 & 18.6 & 18.6 & 22.5 & 22.5 & 27.7 & 27.7 \\
\hline

$4$ & 25.4 & 25.4  & 37.2  & 37.2 & 44.9 & 44.9 & 55.4 & 55.4 \\
\hline

$5$ & 50.7 & 50.7 & 74.3 & 74.3 & 89.8 & 89.8 & 110.7 & 110.7 \\
\hline

$6$ & 101.4 & 101.4 & 148.6 & 148.6 & 179.7 & 179.7 & 221.4 & 221.4 \\
\hline

$7$ & 202.9 &  202.9 & 297.2 & 297.2 & 359.3 & 359.3 & 442.8 & 442.8 \\

\hline
$8$ & 405.8 & 405.8 & 594.4 & 594.4 & 718.7 & 718.7 & 883.3 & 885.6 \\ 
\hline

$9$ &  811.5 &  811.5 & 1188.8 & 1188.8 & 1437.4 & 1437.4 & 1771.2 & 1771.2 
\\

\hline
\end{tabular}\label{tab:res1}
\end{adjustwidth}
\end{center}
\end{table}
}
\endgroup
\begingroup\tiny{
\begin{table}[ht]
\caption{Results for $f_2(x_{n - 1},x_n)$ for odd $q\leq 11$.}
\begin{center}
\begin{adjustwidth}{-0.35in}{-0.35in}
\begin{tabular}{|c||c|c||c|c||c|c||c|c|}\hline              
$q$ & 3 & 3 & 5 & 5 & 7 & 7 & 11 & 11 \\
\hline
$x_1^2=$ & $2x_1+1$ & $2x_1+1$  & $3x_1+2$ & $3x_1+2$ & $x_1+4$ & $x_1+4$ & $4x_1+9$& $4x_1+9$ \\
\hline 
\hline
$n$ & $\log_2$(o($x_n$)) & $\log_2$(o($\delta_n$)) & $\log_2$(o($x_n$)) & $\log_2$(o($\delta_n$)) & $\log_2$(o($x_n$))& $\log_2$(o($\delta_n$)) & $\log_2$(o($x_n$)) & $\log_2$(o($\delta_n$))\\ 
\hline  

$1$ & 3.0 & 3.0 & 4.6 & 4.6 & 5.6 & 5.6 & 6.9 & 4.6 \\\hline

$2$ & 6.3 & 6.3 & 9.3 & 9.3 & 11.2 & 11.2 & 13.8 & 12.3 \\\hline

$3$ & 12.7 & 12.7 & 18.6 & 18.6 & 20.9 & 22.5 & 26.1 & 27.7 \\\hline

$4$ & 25.4 & 25.4  & 37.2  & 37.2 & 44.9 & 44.9 & 55.4 & 48.9 \\\hline

$5$ & 50.7 & 50.7 & 74.3& 74.3 & 88.3 & 89.8 & 106.6 & 110.7 \\\hline

$6$ & 101.4 & 101.4 & 148.6 & 148.6 & 179.7 & 179.7 & 221.4 & 219.8 \\\hline

$7$ &  202.9 &  202.9 & 297.2 & 297.2 & 357.8 & 359.3 & 441.2 & 442.8 \\\hline

$8$ & 405.8 & 405.8 & 594.4 & 594.4 & 718.7 & 718.7 & 885.6 & 879.2 \\ \hline

$9$ &  811.5 &  811.5 & 1188.8 & 1188.8 & 1435.8 & 1437.4 & 1767.1 & 1771.2 \\\hline

\end{tabular}\label{tab:res2}
\end{adjustwidth}
\end{center}
\end{table}}
\endgroup
\begingroup
\tiny{
\begin{table}[ht]
\caption{Results for $f_3(x_{n - 1},x_n)$ for odd $q \leq 11$.}
\begin{center}
\begin{adjustwidth}{-0.35in}{-0.35in} 
\begin{tabular}{|c||c|c||c|c||c|c||c|c|}
\hline             

$q$ & 3 & 3 & 5 & 5 & 7 & 7 & 11 & 11 \\
\hline

$x_1^2=$ & $x_1+1$& $x_1+1$  & $2x_1+2$ & $2x_1+2$ & $3x_1+2$& $3x_1+2$ & $4x_1+9$& $4x_1+9$ \\
\hline 
\hline
$n$ & $\log_2$(o($x_n$)) & $\log_2$(o($\delta_n$)) & $\log_2$(o($x_n$)) & $\log_2$(o($\delta_n$)) & $\log_2$(o($x_n$))& $\log_2$(o($\delta_n$)) & $\log_2$(o($x_n$)) & $\log_2$(o($\delta_n$)) \\ 
\hline  

$1$ & 3.0 & 3.0 & 4.6 & 4.6 & 5.6 & 4.0 & 6.9 & 5.3 \\
\hline

$2$ & 6.3 & 4.0 & 9.3 & 5.6 & 11.2 & 5.0 & 13.8 & 6.3 \\
\hline

$3$ & 12.7 & 5.0 & 18.6 & 6.6 & 22.5 & 6.0 & 25.4 & 7.3 \\
\hline

$4$ & 25.4& 6.0  & 37.2  & 7.6 & 44.9 & 7.0 & 51.3 & 8.3 \\
\hline

$5$ & 50.7 & 7.0 & 74.3 & 8.6 & 89.8 & 8.0 & 106.6 & 9.3 \\
\hline

$6$ & 101.4& 8.0 & 148.6 & 9.6 & 179.7 & 9.0 & 217.3 & 10.3 \\
\hline

$7$ &  202.9 &  9.0 & 297.2 & 10.6 & 359.3 & 10.0 & 436.4 & 11.3 \\
\hline

$8$ & 405.8 & 10.0 & 594.4& 11.6 & 718.7 & 11.0 & 881.5 & 12.3 \\ 
\hline

$9$ &  811.5 &  11.0 & 1188.8 & 12.6 & 1437.4 & 12.0 & 1767.1 & 13.3 \\
\hline

\end{tabular}\label{tab:res3}
\end{adjustwidth}
\end{center}
\end{table}
}
\endgroup
\begingroup \tiny{
\begin{table}[ht]
\caption{Results for $f_4(x_{n - 1},x_n)$ for odd $q \leq 11$.}
\begin{center}
\begin{adjustwidth}{-0.35in}{-0.35in} \begin{tabular}{|c||c|c||c|c||c|c||c|c|}\hline             
$q$ & 3 &3 & 5 &5 & 7 &7 &11 & 11\\
\hline
$x_1^2=$ & $x_1+1$& $x_1+1$  & $4x_1+3$ & $4x_1+3$ & $2x_1+4$& $2x_1+4$ & $7x_1+4$& $7x_1+4$ \\
\hline 
 \hline
$n$ & $\log_2$(o($x_n$)) & $\log_2$(o($\delta_n$)) & $\log_2$(o($x_n$)) & $\log_2$(o($\delta_n$)) & $\log_2$(o($x_n$))& $\log_2$(o($\delta_n$)) & $\log_2$(o($x_n$)) & $\log_2$(o($\delta_n$)) \\ \hline  $1$ & 3.0 & 3.0 & 4.6 & 4.6 & 5.6 & 5.6 & 6.9 & 4.6 \\
\hline
$2$ & 6.3 & 4.0 & 9.3 & 7.7 & 11.2 & 11.2 & 13.8 & 13.8 \\
\hline
$3$ & 12.7 & 5.0 & 17.0 & 17.0 & 22.5 & 20.9 & 27.7 & 27.7 \\
\hline
$4$ & 25.4 & 6.0  & 35.6  & 37.2 & 41.0 & 43.3 & 55.4 & 55.4 \\
\hline$5$ & 50.7 & 7.0 & 72.7 & 72.7 & 89.8 & 89.8 & 110.7 & 109.1 \\
\hline$6$ & 101.4& 8.0 & 147.0 & 148.6 & 177.3 & 179.7 & 221.4 & 219.8 \\
\hline$7$ &  202.9 &  9.0 & 295.6 & 295.6 & 359.3 & 357.8 & 442.8 & 441.2 \\
\hline$8$ & 405.8 & 10.0 & 592.8  & 594.4 & 717.1 & 717.1 & 885.6 & 884.0 \\ \hline$9$ &  811.5 &  11.0 & 1187.2 & 1188.8 & 1435.8 & 1435.8 & 1769.6 & 1771.2 \\
\hline
\end{tabular}
\label{tab:res4}
\end{adjustwidth}
\end{center}
\end{table}}
\endgroup
\begingroup\tiny{
\begin{table}[ht]
\caption{Results for $f_5(x_{n - 1},x_n)$ for odd $q\leq 11$.}
\begin{center}\begin{adjustwidth}{-0.35in}{-0.35in} \begin{tabular}{|c||c|c||c|c||c|c||c|c|}
\hline              $q$ & 3 &3 & 5 &5 & 7 &7 &11 & 11 \\
\hline$x_1^2=$ & $x_1+1$& $x_1+1$  & $x_1+3$ & $x_1+3$ & $2x_1+2$& $2x_1+2$ & $4x_1+4$& $4x_1+4$ \\
\hline 
\hline$n$ & $\log_2$(o($x_n$)) & $\log_2$(o($\delta_n$)) & $\log_2$(o($x_n$)) & $\log_2$(o($\delta_n$)) & $\log_2$(o($x_n$))& $\log_2$(o($\delta_n$)) & $\log_2$(o($x_n$)) & $\log_2$(o($\delta_n$)) \\ \hline  $1$ & 3.0 & 3.0 & 4.6 & 4.6 & 5.6 & 5.6 & 6.9 & 3.0 \\
\hline$2$ & 6.3& 4.0 & 9.3& 5.6 & 8.9 & 5.0 & 13.8& 5.6 \\
\hline$3$ & 12.7& 5.0 & 18.6 & 8.7 & 22.5& 12.2 &27.7& 14.8 \\
\hline$4$ & 25.4& 6.0  & 35.6  & 18.0 & 42.6 & 23.5 & 55.4 & 28.7 \\
\hline$5$ & 50.7 & 7.0 & 72.7 & 38.2 & 89.8 & 45.9 & 110.7 & 56.4 \\
\hline$6$ & 101.4& 8.0 & 147.0 & 73.7 & 179.7 & 89.3 & 221.4 & 110.1 \\
\hline$7$ &  202.9 &  9.0 & 295.6 & 149.6 & 359.3 & 179.1 & 442.8 & 220.8 \\
\hline$8$ & 405.8 & 10.0 & 592.8& 296.6 & 718.7 & 358.8 & 883.3 & 442.8 \\ 
\hline$9$ &  811.5 &  11.0 & 1187.2 & 595.4 & 1437.4 & 718.1 & 1771.2 & 885.0 \\
\hline
\end{tabular}
\label{tab:res5}
\end{adjustwidth}\end{center}
\end{table}}
\endgroup

\begingroup
\tiny{
\begin{table}[ht]
\caption{Upper bounds  for odd $q\leq 11$ and lower bound.}
\begin{center}
\begin{tabular}{|c||c||c|c|c|c|}
\hline 
              
$q$  & 3 & 5 &7  & 11 & Lower bound\\
\hline              

$n$  & $\log_2(q^{2^n}-1)$ & $\log_2(q^{2^n}-1)$ & $\log_2(q^{2^n}-1)$ & $\log_2(q^{2^n}-1)$& $\log_2(2^{(n^2+3n)/2})$ \\
\hline  

$1$  & 3.0 & 4.6 & 5.6 & 6.9 & 2.0\\
\hline

$2$ & 6.3 & 9.3 & 11.2 & 13.8& 5.0  \\
\hline

$3$  & 12.7 &  18.6 & 22.5 & 27.7& 9.0 \\
\hline

$4$ & 25.4  & 37.2 & 44.9  & 55.4 & 14.0 \\
\hline

$5$  & 50.7 & 74.3 & 89.8  & 110.7 & 20.0\\
\hline

$6$  & 101.4 & 148.6 & 179.7 & 221.4 & 27.0\\
\hline

$7$  &  202.9 & 297.2 & 359.3 & 442.8 &  35.0\\
\hline

$8$  & 405.8 & 594.4 & 718.7 & 885.6 & 44.0\\ 
\hline

$9$  &  811.5 & 1188.8 & 1437.4 & 1771.2 &  54.0\\
\hline

\end{tabular}\label{tab:bounds}
\end{center}
\end{table}
}
\endgroup

\clearpage

\begingroup
\tiny{
\begin{table}[ht]
\caption{Results for $f_6(x_{n - 1},x_n)$ and $f_7(x_{n - 1},x_n)$ for $q=2$ and related lower and upper bounds. }
\begin{adjustwidth}{-0.50in}{-0.50in} 
\begin{center}
\begin{tabular}{|c||c|c|c|c|}
\hline 

$f(x_{n-1},x_n)=$ & $f_6(x_{n-1},x_n)$ & $f_7(x_{n-1},x_n)$& Lower bound & Upper bound \\
\hline \hline              

$n$ & $\log_2$(o($x_n$)) & $\log_2$(o($x_n$)) & $\log_2(3^{n(n+3)/2})$ & $\log_2(4^{3^n}-1) $\\ 
\hline  

$1$ & 6.0  & 6.0 &  3.2 & 6.0 \\
\hline  

$2$ & 18.0& 18.0 & 7.9 & 18.0 \\
\hline

$3$ & 54.0  & 54.0 & 14.3 & 54.0\\
\hline

$4$ & 162.0  & 162.0 & 22.2 & 162.0 \\
\hline

$5$ & 486.0   & 486.0 &  31.7 & 486.0  \\
\hline

$6$ & 1458.0   & 1458.0   & 42.8  & 1458.0 \\
\hline

\end{tabular}\label{tab:res6}
\end{center}
\end{adjustwidth}
\end{table}
}
\endgroup

\begingroup
\tiny{
\begin{table}[ht]
\caption{Comparative Analysis}
\begin{center}
\begin{tabular}{|c||c||c|c|c|c|}
\hline 
              
$n$  & $\log_2 (|{{\mathbb{F}}^{*}}_{3^{2^n}}|))$ & Our Model & Burkhart's Model \cite{burkhart2009finite}  & Cohen's Model \cite{cohen1992explicit}  & McNay's Model \cite{mcnay1995topics} \\
              
\hline              

$1$  & 3.0 & 3.0 & 3.0 & 3.0 & 3.0 \\
\hline

$2$ & 6.3 & 6.3 & 5.3 & 4.0 & 4.3 \\
\hline

$3$  & 12.7 &  12.7 & 10.7 & 5.0 & 7.4  \\
\hline

$4$ & 25.4  & 25.4 & 22.4  & 6.0 & 13.7 \\
\hline

$5$  & 50.7 & 50.7 & 46.8  & 7.0 & 26.4 \\
\hline

$6$  & 101.4 & 101.4 & 96.5 & 8.0 & 51.7 \\
\hline

$7$  &  202.9 & 202.9 & 197.0 & 9.0 & 102.4 \\
\hline

$8$  & 405.8 & 405.8 & 399.0 & 10.0 & 203.9 \\ 
\hline

$9$  &  811.5 & 811.5 & 804.0 & 11.0 & 406.8 \\
\hline

\end{tabular}\label{tab:ComparativeAnalysis}
\end{center}
\end{table}
}
\endgroup

\section{Conclusion and future work}

In~\cite{burkhart2009finite}, the choice of polynomials for the recursive process to generate high order elements in finite field extensions, was limited to the equations of the modular curve towers in \cite{elkies2001explicit}. In this work, we attempted to generalize the choice of the polynomials. This provides us with more examples with similar properties. A central theme of this research work is to find a recipe to choose polynomials to use the recursive process. There might be other equations which could help to attain similar bounds. It would be interesting to understand in general which equations are good and which ones are not. We also point out that there could be other explicit towers satisfying similar properties. We were in fact attracted previously by other interesting examples with $v(x)$ being a polynomial of higher degree over $\mathrm{GF}(q, 1)$, which turned out to give high order elements, although the proof seems to be much harder. A possible relation linking together these equations could allow to obtain other families of towers with good parameters. We also expect to improve our results by
extending the construction of Section \ref{sec:odd} to higher degree polynomials and
extending the construction of Section \ref{sec:even} to odd characteristic $q > 3$.

Another question that would be interesting to explore is the possible relation with some geometric construction.
In fact, since the tower in \cite{burkhart2009finite} is obtained from the equation of a modular curve, it is a natural question to ask whether our results have a geometric interpretation or not. We hope that a finer understanding of the subject might also possibly provide a recipe for finding high order elements from towers obtained from different forms.

\section{Acknowledgements}

The second author was partially supported by the research grant "Ing. Giorgio Schirillo" of the Istituto Nazionale di Alta Matematica "F. Severi", Rome. The third author would like to thank the High Performance Computing facility of University of L’Aquila (UAQ), which enabled us to implement the algorithm on MAGMA \cite{MAGMA} computer algebra system,  and run the experiments to validate our results. The third author is grateful to the fruitful and illuminating discussions with Professor Norberto Gavioli, UAQ. The third author is thankful to Professor Kalyan Chakraborty for arranging his research visit to Harish-Chandra Institute (HRI), Prayagraj, India, where the possible modularity aspects of this research work were explored. The third author thanks Dr. Kalyan Banerjee (Post-Doctoral Fellow, HRI) for his interest in exploring the geometric meaning of such towers and the modularity connections.

\bibliography{Ref_Modular_Towers}
\bibliographystyle{amsalpha}

\end{document}